\documentclass[11pt]{article}
\usepackage[utf8]{inputenc}
\usepackage[english]{babel}
\usepackage{amsmath}
\usepackage{amsthm}
\usepackage{amssymb}
\usepackage{enumitem}
\usepackage{authblk}

\newtheorem{theorem}{Theorem}[section]

\newtheorem{corollary}[theorem]{Corollary}
\newtheorem{lemma}[theorem]{Lemma}
\newtheorem{proposition}[theorem]{Proposition}

\theoremstyle{definition}
\newtheorem{definition}[theorem]{Definition}

\theoremstyle{remark} \theoremstyle{remark}
\newtheorem{remark}[theorem]{Remark}
\newtheorem{example}[theorem]{Example}

\numberwithin{equation}{section}

\newcommand{\N}{\mathbb{N}}
\newcommand{\R}{\mathbb{R}}
\newcommand{\C}{\mathbb{C}}
\newcommand{\X}{\mathfrak{X}}
\newcommand{\f}{\varphi}

\begin{document}
	\title{\textbf{A Frankel type theorem for generic submanifolds of Sasakian manifolds}}
	\author{Dario Di Pinto and Antonio Lotta}

\maketitle

\begin{abstract}
	We introduce a weaker notion of generic submanifold of a Sasakian manifold and we prove a Frankel type theorem for this kind of submanifolds under suitable hypotesis on the index of the scalar Levi forms determined by normal directions. It concerns the intersection between a generic and an invariant submanifold and the intersection between two generic submanifolds. From this theorem we derive some topological information about generic submanifolds of Sasakian space forms.
\end{abstract}
\medskip



\section{Introduction}
In this paper we deal with {generic submanifolds} of a Sasakian manifold, and we
establish a sufficient condition for two of them to have non empty intersection,
following Frankel's classical approach, which goes back to \cite{Frankel}. 
We shall also consider the case when one of the submanifolds  is generic and 
the other one is invariant. In \cite{Ornea} and \cite{Pitis} this kind of results were 
discussed for the case of two invariant submanifolds.

The definition of  generic submanifold we shall adopt here is weaker than 
the standard one, see for instance \cite{Yano-Kon2} and \cite{Yano-Kon}. 
Let $(M,\f,\xi,\eta,g)$ be a Sasakian manifold. A submanifold $N$ of $M$ 
will be called {\em generic} provided:

\medskip
a) the Reeb vector field $\xi$ is nowhere normal to $N$;

b) $\f(TN^\perp)\subset TN$, where $TN^\perp$ is the normal bundle of $N$.

\medskip
We remark that in the literature it is customary to  assume $N$  tangent to $\xi$, 
so our assumption is less stringent.

\medskip
Our approach will be based on the fact that such a submanifold is naturally endowed with a $CR$ structure 
$(HN,J)$ of codimension $p+1$,
where $p=\dim(M)-\dim(N)$ (see Prop. \ref{generic_CR_srtucture}). Like the canonical $CR$ structure of the ambient manifold $M$, this induced structure is also strongly pseudoconvex, namely the scalar Levi form $\mathfrak{L}_\eta$ determined by the restriction to $N$ of the contact form $\eta$ is, up to a constat factor, the restriction 
to $HN$ of the Riemannian  metric $g$.  
We refer the reader to \S\ref{prelim} for the definition of the 
(scalar) Levi forms of a $CR$ manifold; here we just recall that each of them
is a Hermitian symmetric bilinear form on the holomorphic 
tangent space $H_xN$ at a point $x\in N$,  intrinsically attached 
to a cotangent vector $\omega\in T_xN^*$ annihilating $H_xN$.
Observe that a) ensures that $\eta$ is everywhere non vanishing on our
submanifold.

We shall focus instead on the Levi forms $\mathfrak{L}_\nu$ determined by
the normal directions $\nu$ to $N$; by definition, 
$\mathfrak{L}_\nu$ is attached to the covector:
$$\omega(X)=g(\varphi\nu,X).$$
Each of these Levi forms will be called {\em characteristic};
we shall denote its index by $i(\mathfrak{L}_\nu)$ and by $n(\mathfrak{L}_\nu)$ its
nullity. 

\bigskip
With this terminology, our main result is the following.

\begin{theorem}\label{main_theorem}
	Let $(M,\f,\xi,\eta,g)$ be a connected, complete  Sasakian manifold with non negative
	$\f$-bisectional curvature. Let $N$ and $P$ submanifolds of $M$, and assume one of them closed and the other compact. 
	
	\medskip
	a) If $N$ is generic and $P$ is invariant, we have 
	$$N\cap P\not=0,$$
	provided for each characteristic Levi form of $N$ it holds:
	\begin{equation}\label{hp(a)}
		i(\mathfrak{L}_\nu)\ge\dim(M)-\dim(P).
	\end{equation}
	
	b) If both $N$ and $P$ are generic submanifolds, set 
	\begin{equation}\label{q,s}
		q:=\min_{\nu\in TN^\perp}i(\mathfrak{L}_\nu) \quad\text{and}\quad s:=\min_{\nu'\in TP^\perp}(i(\mathfrak{L}_{\nu'})+n(\mathfrak{L}_{\nu'})).
	\end{equation}
	Then we have
	$$N\cap P\not=0,$$
	provided $q>0$, $s>0$ and 
	\begin{equation}\label{hp(b)}
		q+s\ge\dim(M)-1.
	\end{equation}
\end{theorem}

\medskip
As an application, we get the following results:

\begin{corollary}\label{cor.1}
	Let $(M,\f,\xi,\eta,g)$ be a complete, connected, regular Sasakian manifold with nonnegative $\f$-bisectional curvature and assume that $M$ fibers onto a K\"ahler manifold biholomorphic to a product $S\times\mathbb{C}$, where $S$ is a complex manifold. 
	
	Then every generic submanifold $N$ of $M$ whose characteristic Levi forms have all positive index is non compact.   
\end{corollary}

\medskip
\begin{corollary}\label{cor.2}
	Let $N$ be a generic submanifold of the Sasakian space form $\mathbb{S}^{2n+1}(c)$ with $\varphi$-sectional
	curvature $c$, $c>-3$. If all the characteristic Levi forms of $N$ 
	have positive index, then $N$ cannot be contained in an open hemisphere.
\end{corollary}

\section{Preliminaries}
\label{prelim}
Let's start by recalling the definitions of $CR$ manifolds, Levi-Tanaka forms and scalar Levi forms.
In the following, given a vector bundle $E$ over a smooth differential manifold $M$, we will denote by $\Gamma(E)$ the $\mathcal{C}^\infty(M)$-module of global smooth sections of $E$.\\

Let $M$ be a smooth real differential manifold of dimension $n$, and let $m,k\in\N$ such that $2m+k=n$. If $HM$ is a real vector subbundle of rank $2m$ of the tangent bundle $TM$ and $J:HM\to HM$ is a bundle isomorphims such that $J^2=-Id$, the couple $(HM,J)$ is said to be a \textit{CR structure} on $M$ if the following properties hold for all $X,Y\in\Gamma(HM)$ smooth section of $HM$:
\begin{enumerate}[label=(\roman*)]
	\item $[JX,JY]-[X,Y]\in\Gamma(HM)$;
	\item $N_J(X,Y):=[JX,JY]-J[JX,Y]-J[X,JY]-[X,Y]=0$.
\end{enumerate}
In this case $(M,HM,J)$ is called a \textit{CR manifold of type (m,k)} and $m,k$ are the \textit{CR dimension} and the \textit{CR codimension} of the $CR$ structure, rispectively.\\

\begin{definition}
	Let $(M,HM,J)$ be a $CR$ manifold of type $(m,k)$. Given a point $x\in M$, the \textit{Levi-Tanaka form of $M$ at $x$} is the bilinear map $$L_x:H_xM\times H_xM\to T_xM/H_xM$$ defined by \begin{equation}
	L_x(X,Y):=\pi_x([\tilde{X},J\tilde{Y}]_x)\quad \forall X,Y\in H_xM,
	\end{equation}
	where $\tilde{X},\tilde{Y}\in\Gamma(HM)$ are two arbitrary extensions of $X,Y$ and $\pi:TM\to TM/HM$ is the canonical projection on the quotient bundle $TM/HM$.
\end{definition}
It is known that $L_x$ is well defined because the value $\pi_x([\tilde{X},J\tilde{Y}]_x)$ only depends on the values of $\tilde{X},\tilde{Y}$ at $x$, that is on $X$ and $Y$.\\
Moreover, according to (i) above, $L_x$ turns to be a vector valued symmetric Hermitian form on the holomorphic tangent space $H_xM$ with respect to the complex structure $J:=J_x$, that is
\begin{equation}
L_x(X,Y)=L_x(JX,JY),\quad L_x(X,Y)=L_x(Y,X)
\end{equation}
for all $X,Y\in H_xM$.

Given a point $x$ on the $CR$ manifold $(M,HM,J)$, we will denote by $$H^0_xM:=\{\omega\in T^*_xM\ |\ \omega(X)=0\quad \forall X\in H_xM\}$$
the annihilator of $H_xM\subset T_xM$. Then we have the following definition.
\begin{definition}
	Let $(M,HM,J)$ be a $CR$ manifold, $x\in M$ and $\omega\in H^0_xM$. The Hermitian form \begin{equation}
	\mathfrak{L}_\omega:H_xM\times H_xM\to\R\quad\text{s.t.}\quad \mathfrak{L}_\omega(X,Y):=\omega L_x(X,Y)
	\end{equation} 
	is called the \textit{scalar Levi form determined by $\omega$ at $x$}.
\end{definition}
\begin{remark}
	Since the scalar Levi forms $\mathfrak{L}_\omega$ are symmetric, it makes sense to consider their index $i(\mathfrak{L}_\omega)$, defined as the minimum between the number of positive and negative eigenvalues of $\mathfrak{L}_\omega$.
\end{remark}

More specifically we recall the following terminology from $CR$ geometry; see for instance \cite{Nacinovich-Medori}.

\begin{definition}
	Let $(M,HM,J)$ be a $CR$ manifold of type $(m,k)$ and let $x\in M$.\\
	$M$ is said \textit{pseudoconvex} at $x$ if $\mathfrak{L}_\omega$ is positive definite for some $\omega\in H^0_xM$.
	If there exists a global section $\omega\in\Gamma(H^0M)$ such that $\mathfrak{L}_\omega$ is positive definite at each point $x\in M$, $M$ is called \textit{strongly pseudoconvex}.\\
	$M$ is said \textit{pseudoconcave} at $x$ if $i(\mathfrak{L}_\omega)>0$ for every $\omega\in H_x^0M$, $\omega\neq0$.\\
	$M$ is said \textit{Levi-flat at $x$} if $\mathfrak{L}_\omega=0$ for every $\omega\in H^0_xM$, i.e. $L_x\equiv0$.  
\end{definition}

We close this section by recalling that a Sasakian manifold $(M,\f,\xi,\eta,g)$, as defined in \cite{Blair}, is a particular kind of strongly pseudoconvex $CR$ manifold of hypersurface type, i.e. of $CR$ codimension 1. We shall refer to \cite{Blair} for the notation and basic facts concerning Sasakian geometry. We only remark that in this case the $CR$ structure is given by the contact distribution $\mathcal{D}=\ker\eta=\left<\xi\right>^\perp$ and the almost complex structure is $J=\f_{|\mathcal{D}}$. Therefore, for any $x\in M$, $H^0_xM$ is spanned by $\eta_x$ and, up to scaling, we have only one scalar Levi form $\mathfrak{L}_{\eta_x}$. Moreover, since $M$ is a contact metric manifold, the identity $$d\eta(X,Y)=g(X,\f Y)$$ yields that $$\mathfrak{L}_{\eta_x}=2{g_x}_{| H_xM\times H_xM}.$$

\section{Invariant and generic submanifolds of Sasakian manifolds}
Let $(M,\f,\xi,\eta,g)$ be a Sasakian manifolds. We shall consider two kind of submanifolds of $M$: the invariant submanifolds and the generic submanifolds, defined as follows.

\begin{definition}
	An \textit{invariant submanifold} of a Sasakian manifold $(M,\f,\xi,\eta,g)$ is a real submanifold $N$ of $M$ such that $\dim N<\dim M$ and $\f(TN)\subset TN$.
\end{definition}

It is known that invariant submanifolds are always tangent to $\xi$ and they inherit a Sasakian structure from the Sasakian structure of the ambient manifold by restriction (\cite{Blair}). Thus invariant submanifolds have odd dimension.
Moreover, the second fundamental form $\alpha$ of $N$ satisfies the following property:
\begin{equation}\label{invariant subm./sec.f.f.}
\alpha(X,X)+\alpha(\f X,\f X)=0\quad\forall X\in\X(N).
\end{equation}

\smallskip
\begin{definition}
	A \textit{generic submanifold} of a Sasakian manifold $(M,\f,\xi,\eta,g)$ is a submanifold $N$ such that $\xi$ is nowhere normal to $N$ and $\f(TN^\perp)\subset TN$, where $TN^\perp$ is the normal bundle of $N$.
\end{definition}

\begin{example}\label{Rmk:Hypersurface}
	Every hypersurface $N$ with $\xi$ nowhere normal to $N$ is a generic submanifold of $M$. 
\end{example}

If $N$ is a generic submanifold of a Sasakian manifold $(M,\f,\xi,\eta,g)$, the component $\zeta$ of $\xi$ tangent to $N$ is always nonzero and we can consider the following orthogonal decomposition of the tangent bundle of $N$:
\begin{equation}
TN=\left<\zeta\right>\oplus\left<\zeta\right>^\perp=\left<\zeta\right>\oplus\f(TN^\perp)\oplus HN,
\end{equation}
where $\f(TN^\perp)$ is a subbundle of $TN$ because for every $\nu\in TN^\perp$, $\nu\neq0$, we always have that $\f\nu\neq0$, since $\xi$ is nowhere normal to $N$.\\ Note that in each point $x\in N$ $H_xN$ is given by 
\begin{equation}
H_xN=\{X\in T_xN\ |\ X\perp\zeta_x\ \text{and}\ \f X\in T_xN\}.
\end{equation} 

\begin{remark}\label{HN_normal_to_xi}
	It is clear that for any $X\in H_xN$, we have $\eta_x(X)=0$, since $X$ is both tangent to $N$ and normal to $\zeta$. However this is not true in general for vectors which are normal to $N$, unless $\xi$ is everywhere tangent to $N$. 
\end{remark}

Let us consider the restriction $J:=\f_{|HN}:HN\to HN$ of $\f$ to $HN$. $J$ is well defined bacause $\f(HN)\subset HN$.  Indeed, for all $x\in N$ and for all $X\in H_xN$, we have $$g(\f X,\zeta_x)=g(\f X,\xi_x)=0,$$ and for all $\nu\in T_xN^\perp$: $$g(\f X,\f\nu)=g(X,\nu)-\eta(X)\eta(\nu)=0,$$ by the previous remark. Thus $\f X$ is orthogonal to $\f(T_xN^\perp)$ and therefore $\f X\in H_xN$.\\
Moreover, since $HN\subset\left<\xi\right>^\perp$, we have that $J^2=-Id$, which means that $J$ is an almost complex structure on $HN$.

\begin{proposition}\label{generic_CR_srtucture}
	Let $(M,\f,\xi,\eta,g)$ be a Saskian manifold of dimension $2n+1$ and let $N\subset M$ be a generic submanifold of codimension $p$. Then the couple $(HN,J)$ defines a $CR$ structure on $N$ of $CR$ codimension $p+1$.
\end{proposition}
\begin{proof}
	First we recall that Sasakian manifolds are characterized by means of the following identity, involving the covariant derivative of $\varphi$ with respect to the Levi-Civita connection (see \cite{Blair}):
	\begin{equation}\label{Sasaki}
		(\nabla_X\f)Y=g(X,Y)\xi-\eta(Y)X.
	\end{equation}
	Now, fix $x\in N$, $X,Y\in H_xN$ and take two smooth section in $\Gamma(HN)$ which extend $X,Y$. Since $X,Y$ are normal to $\xi_x$, $\eta(X)=\eta(Y)=0$ so that $d\eta(X,Y)=-\frac12\eta[X,Y]$. The same holds for $JX$ and $JY$. 
	Thus we have 
	$\eta([X,Y]-[JX,JY])=0$, namely $[X,Y]-[JX,JY]$ is normal to $\xi_x$. But since $[X,Y]-[JX,JY]$ is tangent to $N$ this means that $[X,Y]-[JX,JY]$ is normal to $\zeta_x$.\\
	Now, by using (\ref{Sasaki}), 
	we have: 
	\begin{eqnarray}\label{eqn1}
		&&\f([JX,JY]-[X,Y])=\nonumber\\
		&=&\f(\nabla_{\f X}{\f Y}-\nabla_{\f Y}{\f X})-\f(\nabla_XY-\nabla_YX)= \nonumber \\
		&=&-\nabla_{\f X}Y+\nabla_{\f Y}X-\f\nabla_XY+\f\nabla_YX=\nonumber\\
		&=&-\nabla_{\f X}Y+\nabla_{\f Y}X+(\nabla_X\f)Y-\nabla_X\f Y - (\nabla_Y\f)X+\nabla_Y\f X=\nonumber\\
		&=&[Y,\f X]+[\f Y,X]-\eta(Y)X+\eta(X)Y=\nonumber\\
		&=&-[\f X,Y]-[X,\f Y],
	\end{eqnarray}
	which is tangent to $N$ since so are $X,\f X,Y,\f Y$. \\
	Thus we have proved that $[X,Y]-[JX,JY]\in H_xN$. Finally, rewriting (\ref{eqn1}) as $$J([JX,JY]-[X,Y])=-[JX,Y]-[X,JY],$$ by applying $J$ and by using $J^2=-Id$, it follows that $N_J=0$.
\end{proof}

Thanks to this result, for each point $x\in N$ it make sense to consider the Levi-Tanaka form at $x$ and the scalar Levi forms $\mathfrak{L}_\omega$ where $\omega$ varies in $H_x^0N$. \\
In particular, given a non zero normal direction $\nu\in T_xN^\perp$, consider the 1-form $\omega:T_xN\to\R$ such that 
\begin{equation}
\omega(X):=g(X,\f\nu)\quad\forall X\in T_xN.
\end{equation}
Clearly $\omega\neq0$ because of $\f\nu\in T_xN$ and $\omega(\f\nu)=\|\f\nu\|^2\neq0$. Furthermore $\omega\in H_x^0N$ by definition of $H_xN$.\\
The scalar Levi form $\mathfrak{L}_\omega$ determined by $\omega$ 
will be denoted by $\mathfrak{L}_\nu$; for convenience, we shall adopt the following terminology:

\begin{definition}
	For any $x\in N$ the scalar Levi forms $\mathfrak{L}_\nu$ determined by $\nu\in T_xN^\perp$, $\nu\neq0$, will be called \textit{characteristic Levi forms} of $N$ at $x$.
\end{definition}

The following proposition establishes a relationship between the second fundamental form of a generic submanifold $N$ and its characteristic Levi forms.

\begin{proposition}\label{scalar-Levi-form/second-fund.form}
	Let $(M,\f,\xi,\eta,g)$ be a Sasakian manifold and let $N\subset M$ be a generic submanifold with second fundamental form $\alpha$. Given $x\in N$ and $0\neq\nu\in T_xN^\perp$, one has \begin{equation}
	\mathfrak{L}_\nu(X,X)=g_x(\alpha(X,X)+\alpha(\f X,\f X),\nu)
	\end{equation}
	for every $X\in H_xN$.
\end{proposition}
\begin{proof} 
	Fix $x\in N$, $X\in H_xN$ and consider a smooth section in $\Gamma(HN)$ which extends $X$.
	Then $\varphi X$ is again tangent to $N$.
	Using the fact that $X$, $\varphi X$ and $\varphi\nu$ are all orthogonal to $\xi$ and identity (\ref{Sasaki}), we get:
	\begin{eqnarray*}
		&&\mathfrak{L}_\nu(X,X)=\\
		&=&g_x([X,\varphi X],\varphi\nu)=\\
		&=&g_x(\nabla_X\varphi X,\varphi\nu)-g_x(\nabla_{\varphi X}X,\varphi\nu)=\\
		&=&g_x(\varphi\nabla_X X,\varphi\nu)+g_x(\varphi\nabla_{\varphi X}X,\nu)=\\
		&=&g_x(\nabla_X X,\nu)+g_x(\nabla_{\varphi X}{\varphi X},\nu)=\\	
		&=&g_x(\alpha(X,X)+\alpha(\f X,\f X),\nu).
	\end{eqnarray*}   
\end{proof}

In the following we will deal with generic and invariant submanifolds of Sasakian manifolds with nonnegative $\f$-bisectional curvature. So we recall the definition of this kind of curvature which was introduced by Tanno and Baik in \cite{Tanno-Baik} and used by in \cite{Ornea} and \cite{Pitis} to obtain Frankel type theorems about the intersection of two invariant submanifolds. It is an adaptation to the Sasakian case of the notion of holomorphic bisectional curvature introduced by Goldberg and Kobayashi in \cite{Golgberg-Kobayashi} for K\"ahler manifolds.

\begin{definition}\label{fi-bisectional curvature}
	Let $(M,\f,\xi,\eta,g)$ be a Sasakian manifold. We say that $M$ has  \textit{nonnegative $\f$-bisectional curvature} if 
	 \begin{equation}
		H(X,Y):=K(X,Y)+K(X,\f Y)\ge 0
	\end{equation}
	for any $x\in M$ and for any $X,Y\in T_xM$ such that $X,Y,\f Y,\xi_x$ are mutually orthonormal, where $$K(X,Y):=R(X,Y,X,Y)=g(R(X,Y)Y,X)$$ denotes the sectional curvature at $x$ of the 2-plane $\left<X,Y\right>\subset T_xM$, and similarly for $K(X,\f Y)$.\\ 
\end{definition}

\begin{proposition}\label{Prop:curvature}
	Let $(M,\f,\xi,\eta,g)$ be a Sasakian manifold with nonnegative $\f$-bisectional curvature. Let $x\in M$ and $X,W\in T_xM$ such that $\eta(X)=0$ and $X,\f X,W$ are mutually orthonormal. Then one has:
	$$R(X,W,X,W)+R(\f X, W,\f X,W)\ge0.$$
\end{proposition}
\begin{proof}
	Here $W$ might not be normal to $\xi_x$. However, by decomposing $W$ as $W=Y+Z$ with $Y,Z\in T_xM$, $Y\perp\xi_x$ and $Z=\lambda\xi_x$,  we get:
	\begin{eqnarray}\label{eq:curv1}
		R(X,W,X,W)
		&=&R(X,Y,X,Y)+R(X,Z,X,Z)+2g(R(X,Y)Z,X)=\nonumber\\
		&=&R(X,Y,X,Y)+R(X,Z,X,Z),
	\end{eqnarray}
	where we have used the following curvature characterization of Sasakian manifolds (see \cite{Blair}):
	$$R(X,Y)\xi=\eta(Y)X-\eta(X)Y\quad \forall X,Y\in\X(M).$$
	From this formula we also get that $R(X,Z,X,Z)$ is nonnegative. Similarly, 
	\begin{equation}\label{eq:curv2}
		R(\f X,W,\f X,W)=R(\f X,Y,\f X,Y)+R(\f X,Z,\f X,Z),
	\end{equation}
	where $R(\f X,Z,\f X,Z)\ge0$.
	Thus, by adding the identities (\ref{eq:curv1}) and (\ref{eq:curv2}), we have: $$R(X,W,X,W)+R(\f X,W,\f X,W)\ge0,$$ since, up to scaling, it is the sum of a $\f$-bisectional curvature and two nonnegative terms.
\end{proof}

\section{Examples}
In this section we wish to exhibit some examples of Sasakian manifolds with nonnegative $\f$-bisectional curvature and of generic submanifolds whose characteristic Levi forms have all positive index.\\

Let's consider a regular Sasakian manifold $(M,\f,\xi,\eta,g)$ of dimension $2n+1$ which fibers on a K\"{a}hler manifold $(N,J,g')$ of dimension $2n$. Then it is given a Riemannian submersion $\pi:M\to N$ whose fibers are 1-dimensional submanifolds of $M$ tangent to $\xi$. Hence the vertical distribution $\mathcal{V}$ is locally spanned by $\xi$: 
\begin{equation}
\mathcal{V}_x=T_x\pi^{-1}(p)=\ker(d\pi)_x=\left<\xi_x\right>
\end{equation} 
for all $x\in M$ such that $\pi(x)=p\in N$. Moreover the differential $d\pi$ commutes with the tensor $\f$ and the complex structure $J$: \begin{equation}\label{J,fi}
d\pi\circ\f=J\circ d\pi.
\end{equation}

In \cite{Tanno-Baik} it is shown that, given a point $x\in M$ and two tangent vector $X,Y\in T_xM$ such that $X,Y,\f Y,\xi_x$ are mutually orthonormal, the $\f$-bisectional curvature $H(X,Y)$ is given by $$H(X,Y)=H'(X',Y'),$$ where $H'(X',Y')$ is the holomorphic bisectional curvature of $N$ at $p=\pi(x)$, related to the vectors $X':=(d\pi)_xX$ and $Y':=(d\pi)_xY$. Hence, if $H'$ is nonnegative, so is the $\f$-bisectional curvature $H$ of $M$.

For instance, Takahashi's globally $\f$-symmetric spaces are examples of regular Sasakian manifolds and  from Theorem 6.4 in \cite{Takahashi} it follows that those of compact type have nonnegative $\f$-bisectional curvature. \\

Now we consider the same Riemannian submersion $\pi:M\to N$ to construct examples of generic submanifolds of the Sasakian manifold $M$.\\
Let's point out that, since $\pi$ is a surjective submersion, if $S'\subset N$ is a submanifold of $N$, then $S:=\pi^{-1}(S')$ is a submanifold of $M$ with the same codimension of $S'$. Moreover, since $$T_xS=(d\pi)_x^{-1}(T_{\pi(x)}S')\quad\forall x\in S,$$ we have that $S$ is tangent to $\xi$.
In particular, if $S'$ is a real hypersurface of $N$, $S$ is a hypersurface of $M$ tangent to $\xi$ and so $S$ is a generic submanifold of $M$, due to Remark \ref{Rmk:Hypersurface}.\\
We want to show that if $S'$ is a psedoconcave hypersurface of $N$, then $S$ is a generic submanifold of $M$, whose characteristic Levi forms have positive index. \\
First of all we recall that the canonical $CR$ structure on a hypersurface $S'$ of the K\"{a}hler manifold $(N,J,g')$ is given by 
\begin{equation}
HS':=TS'\cap J(TS').
\end{equation}
Then, by using identity (\ref{J,fi}) and $\f^2_{|HS}=-Id$, it follows that $(d\pi)_xX\in H_pS'$ for any $x\in S$ such that $\pi(x)=p\in S'$ and for any $X\in H_xS$.\\
Now, let $x\in S$, $p=\pi(x)\in S'$, $X\in H_xS$ and set $X':=(d\pi)_xX\in H_pS'$. 
For any non zero normal vector $\nu\in T_xS^\perp$, from the definition of the characteristic Levi form $\mathfrak{L}_\nu$ and from an elementary property of Riemannian submersions (see for instance Proposition 1.1 in \cite{Falcitelli-Ianus-Pastore}), we get:
\begin{eqnarray*}
	\mathfrak{L}_\nu(X,X)&=&g_x([X,\f X],\f \nu)=\\
	&=&g_x(h[X,\f X],\f\nu)=\\
	&=&g'_p([X',JX'],J\nu')=\\
	&=&\mathfrak{L}'_{\nu'}(X',X'),
\end{eqnarray*} 
where $h[X,\f X]$ is the horizontal component of $[X,\f X]$, $\nu'=(d\pi)_x\nu\in(T_pS')^\perp$ and $\mathfrak{L}'_{\nu'}$ is the scalar Levi form on $H_pS'$ determined by the covector $$\omega'(X')=g'_p(X',J\nu')\quad \forall X'\in T_pS'.$$
In conclusion we have proved that, for every real hypersurface $S'$ of  $N$:
\begin{equation}\label{L=L'}
\mathfrak{L}_\nu(X,X)=\mathfrak{L}_{\nu'}(X',X')\quad \forall X\in H_xS.
\end{equation}
In particular, this implies that all the Levi forms $\mathfrak{L}_\nu$ have positive index if $S'$ is pseudoconcave.

\section{Proofs of the results}
In this section we give the proof of our Frankel type theorem. We begin by proving a lemma which provides us a way to construct an orthonormal set $\{E,\f E\}$ consisting of parallel vector fields along a geodesic $\gamma$.  

\begin{lemma}\label{Lemma:Etilde=fE}
	Let $N$ be a generic submanifold of a Sasakian manifold	\\ $(M,\f,\xi,\eta,g)$ and let $\gamma:[0,l]\to M$ be a geodesic starting from $x\in N$ and orthogonal to $N$ at $x$. If $e\in H_xN$ and $E,\tilde{E}\in\X(\gamma)$ are obtained by parallel translation of $e,\f e$ respectively along $\gamma$, then $E$ is orthogonal to $\xi$ along $\gamma$ and $\tilde{E}=\f E$.
\end{lemma}
\begin{proof}
	To prove that $\tilde{E}=\f E$ we need to show that $\f E$ is parallel along $\gamma$ and for this purpose we follow the same idea of a proof in \cite{Ornea}.\\ Since $e\in H_xN$ is tangent to $N$, while $\dot\gamma(0)$ is normal to $N$, we have that $e\perp\dot\gamma(0)$. Moreover, since $E$ and $\dot\gamma$ are parallel along $\gamma$ and the parallel transport is a linear isometry, we have that $E\perp\dot\gamma$ along $\gamma$, that is $g(E,\dot\gamma)(t)=0$ for every $t\in[0,l]$. Hence:
	\begin{equation}
	(\nabla_{\dot\gamma}\f)E=g(\dot\gamma,E)\xi-\eta(E)\dot\gamma=-\eta(E)\dot\gamma.
	\end{equation}
	On the other hand, \begin{equation}
	(\nabla_{\dot\gamma}\f)E=\nabla_{\dot\gamma}\f E-\f(\nabla_{\dot\gamma}E)=\nabla_{\dot\gamma}\f E, 
	\end{equation}
	since $E$ is parallel along $\gamma$. Therefore $\nabla_{\dot\gamma}\f E=-\eta(E)\dot\gamma$ and to prove the our claim, we just have to prove that $\eta(E)(t)=0$ for every $t\in [0,l]$. From this will also follow that $E$ is normal to $\xi$ along $\gamma$. So we define the function $$f:[0,l]\to\R\quad\text{s.t.}\quad f(t):=\eta(E)(t)=g_{\gamma(t)}(E(t),\xi_{\gamma(t)})$$ and we prove that $f$ is identically zero. To this aim, we note that 
	\begin{eqnarray}
	f'=g(\nabla_{\dot\gamma}E,\xi)+g(E,\nabla_{\dot\gamma}\xi)=-g(E,\f\dot\gamma)=g(\f E,\dot\gamma);\\
	f''=g(\nabla_{\dot\gamma}\f E,\dot\gamma)+g(\f E,\nabla_{\dot\gamma}\dot\gamma)=-\eta(E)g(\dot\gamma,\dot\gamma)=-\|\dot\gamma\|^2f,
	\end{eqnarray}
	where $c:=\|\dot\gamma\|^2\in\R$ is constant. Moreover, $f(0)=g_x(e,\xi_x)=0$ by Remark \ref{HN_normal_to_xi} and $f'(0)=g_x(\f e, \dot\gamma(0))=0$ because of $e,\f e\in H_xN\subset T_xN$ and $\dot\gamma(0)\in T_xN^\perp$.\\
	In conclusion, we proved that $f$ is a solution of the following Cauchy problem $$\left\{\begin{array}{c}
	f''+cf=0\\f(0)=0\\f'(0)=0
	\end{array}\right.$$ so that $f=0$.
\end{proof}
\medskip

\textbf{Proof of Theorem \ref{main_theorem}.}
	We prove the theorem discussing simultaneously the two cases (a) and (b) .\\
	Assume by contradiction that $N\cap P\neq\emptyset$. Thanks to the topological assumptions on the submanifolds, there exist two points $x\in N$ and $y\in P$ such that $l:=d(x,y)=d(N,P)>0$. Moreover, by the completeness of $M$, there exists a length minimizing geodesic $\gamma:[0,l]\to M$, parametrized by arc length, joining $x$ and $y$ and intersecting orthogonally $N$ and $P$. 
	Set $\nu:=\dot\gamma(0)\in T_xN^\perp$ and $\nu':=\dot\gamma(l)\in T_yP^\perp$.\\
	In the case (a), since $\dim M>\dim P$, the index $q:=i(\mathfrak{L}_\nu)$ is strictly positive by (\ref{hp(a)}) and then there exists a linear subspace $V\subset H_xN$ of dimension $q$ on which $\mathfrak{L}_\nu$ is negative definite. Moreover, set $W:=\left<\xi_y\right>^\perp\cap T_yP$ and note that $s:=\dim W=\dim P-1$, since $P$ is tangent to $\xi$.\\
	Similarly, in the case (b), since $q,s>0$, where $q,s$ are as in \eqref{q,s}, there are two linear subspaces $V\subset H_xN$ and $W\subset H_yP$ of dimensions $q$ and $s$ respectively such that $\mathfrak{L}_\nu$ is negative definite on $V$ and $\mathfrak{L}_{\nu'}$ is positive semi-definite on $W$.\\
	Now, in both cases, let us denote by $V'\subset T_yM$ the image of $V$ under the parallel transport along $\gamma$: since $\dot\gamma(0)$ is normal to $N$ and $V\subset H_yN\subset\left<\xi_x\right>^\perp$, using Lemma \ref{Lemma:Etilde=fE}, we see that $V'\subset\left<\dot\gamma(l),\xi_y\right>^\perp$. Moreover, both in (a) and (b), we also have $W\subset\left<\dot\gamma(l),\xi_y\right>^\perp$. Therefore $$V'+W\subset\left<\dot\gamma(l),\xi_y\right>^\perp$$ and $$\dim(V'+W)\le\dim M-2.$$
	Now, by using either (\ref{hp(a)}) for the case (a) or (\ref{hp(b)}) for the case (b), we have: 
		$$\dim(V'+W)\ge q+s-\dim M+2\ge 1.$$
	
	Hence we can consider a non zero vector $e'\in V'\cap W$, which is the imagine of a vector $e\in V$ under parallel translation. In other words there exists a vector field $E\in\X(\gamma)$ which is parallel along $\gamma$ and such that $E(0)=e$, $E(l)=e'$. From Lemma \ref{Lemma:Etilde=fE} it follows that $\f E$ is parallel along $\gamma$ and $\f E(l)=\f e'\in T_yP$ by the invariance of $P$ in the case (a) and by the $\f$-invariance of $H_yP\subset T_yP$ in the case (b).\\
	Computing the index form $I$ of $\gamma$ (see for instance \cite{Sakai}) on the vector fields $E$ and $\f E$, we have:
	
	$$I_0^l(E,E)=-\int_{0}^{l}R(E,\dot\gamma,E,\dot\gamma)dt+g(\alpha(E,E),\dot\gamma)|_0^l;$$
	$$I_0^l(\f E,\f E)=-\int_{0}^{l} R(\f E,\dot\gamma,\f E,\dot\gamma)dt+ g(\alpha(\f E,\f E),\dot\gamma)|_0^l.$$ 
	Finally, by adding this last two expressions and by using Proposition \ref{scalar-Levi-form/second-fund.form}, we get:
	\begin{eqnarray}\label{I+I}
	&&I_0^l(E,E)+I_0^l(\f E,\f E)=\nonumber\\ &=&-\int_0^l[R(E,\dot\gamma,E,\dot\gamma)+R(\f E,\dot\gamma,\f E,\dot\gamma)]dt +\mathfrak{L}_\nu(e,e),
	\end{eqnarray} 
	in the case (a) by the identity (\ref{invariant subm./sec.f.f.}), and 
	\begin{eqnarray}\label{I+I1}
	&&I_0^l(E,E)+I_0^l(\f E,\f E)=\nonumber\\ &=&-\int_0^l[R(E,\dot\gamma,E,\dot\gamma)+R(\f E,\dot\gamma,\f E,\dot\gamma)]dt+ \nonumber\\&& +\mathfrak{L}_\nu(e,e)-\mathfrak{L}_{\nu'}(e',e'),
	\end{eqnarray} 
    in the case (b). Moreover, in view of Proposition \ref{Prop:curvature} and $e\in V$, we conclude that the expression (\ref{I+I}) is strictly negative. Similarly the expression \eqref{I+I1} is also strictly negative since $e'\in W$. In both cases this contradicts the length minimizing property of $\gamma$.
	\qed
\medskip

\begin{remark}
	In the same setting of Theorem \ref{main_theorem} (b), if $N$ is a compact generic submanifold whose characteristic Levi forms have all positive index, then $N$ intersects every closed, totally geodesic and generic hypersurface $P$. \\
	Indeed, since characteristic Levi forms are Hermitian and symmetric, $i(\mathfrak{L}_\nu)>0$ is equivalent to $i(\mathfrak{L}_\nu)\ge 2$ and hence $q\ge2$. Moreover, since $P$ is totally geodesic, the scalar Levi forms $\mathfrak{L}_{\nu'}$, $\nu'\in TP^\perp$, are all identically zero and hence $$s=\dim H_yP=\dim P-2=\dim M-3.$$ Therefore $q+s\ge\dim M-1$ and the claim is proved by applying the theorem.
\end{remark}

\medskip
Finally we present the proofs of the Corollary \ref{cor.1} and Corollary \ref{cor.2}.\\ 

\textbf{Proof of Corollary \ref{cor.1}.} Consider the fibration $\pi:M\to M/\xi$ and, by contradiction, assume that $M$ admits
a compact generic submanifold $N$,  whose characteristic Levi forms have all positive index. Then $\pi(N)$ is a compact set.
Since $M/\xi$ is biholomorphic to $S\times\C$, one can always find a Levi flat real hypersurface $P$ in it, such that $\pi(N)\cap P=\emptyset$ (to see this, it suffices to consider $S\times E$, where $E\subset\C$ is a real straight line disjoint from $p(\pi(N))$, where $p:S\times\C\to\C$ is the natural projection).\\
Then $N\cap\pi^{-1}(P)=\emptyset$, but this is in contrast with b) of Theorem \ref{main_theorem}, because $\pi^{-1}(P)$ is a generic hypersurface of $M$ whose characteristic Levi forms all vanish by \eqref{L=L'}.\qed

\medskip
\begin{remark}
This last corollary can be applied to the Sasakian space form $M(-3)=\R^{2n+1}$. In fact it is a complete, connected Sasakian manifold which has constant $\f$-bisectional curvature equal to zero, fibering onto the complex Euclidean space $\C^n$.
\end{remark}

\medskip
	\textbf{Proof of Corollary \ref{cor.2}.} Recall that, as a manifold, $\mathbb{S}^{2n+1}(c)$ is the
	unit sphere $\mathbb{S}^{2n+1}\subset\R^{2n+2}=\mathbb{C}^{n+1}$, where we adopt the following notation:
	$$(z_1,\dots,z_{n+1})=(x_1,\dots,x_{n+1},x_{n+2},\dots,x_{2n+2}),\quad z_k=x_k+ix_{n+1+k}.$$
	Moreover, the Sasakian structure on $\mathbb{S}^{2n+1}(c)$ is obtained by applying a $\mathcal{D}$-homothetic deformation to the canonical Sasakian structure of $\mathbb{S}^{2n+1}$ and this deformed structure is invariant under the action of the unitary group $U(n+1)$. As a consequence, unitary transformations on $\mathbb{S}^{2n+1}(c)$ map open hemispheres in open hemispheres and generic submanifolds into generic ones, preserving the index of all characteristic Levi forms. Therefore, given a generic submanifold $N\subset\mathbb{S}^{2n+1}(c)$ as in the statement, without loss of generality it suffices to prove that $N$ cannot be contained in the open  hemisphere $$S:=\{x\in\mathbb{S}^{2n+1}(c)|\ x_{2n+2}>0\}.$$ 
	Let $\pi:\mathbb{S}^{2n+1}(c)\to\mathbb{C}$P$_n$ be the canonical projection and let us consider the hyperplane $\sigma:z_{n+1}=0$ of $\mathbb{C}$P$_n$: since $\sigma$ is a holomorphic submanifold, $\pi^{-1}(\sigma)$ is an invariant submanifold of $\mathbb{S}^{2n+1}(c)$. 
	Furthermore, since the characteristic Levi forms are Hermitian and symmetric we have that $i(\mathfrak{L}_\nu)>0$ is equivalent to $$i(\mathfrak{L}_\nu)\ge2=\dim M-\dim\pi^{-1}(\sigma).$$
	Finally, since $c>-3$, $\mathbb{S}^{2n+1}(c)$ has nonnegative $\f$-bisectional curvature and by applying Theorem \ref{main_theorem} (a), we have that $N\cap \pi^{-1}(\sigma)\neq\emptyset$.  This means that there exists a point $P\in N$ with coordinates $P(z_1,\dots,z_n,0)$; in particular $P\notin S$.   
	\qed

Dario Di Pinto\\
Dipartimento di Matematica\\ Universit\`a degli Studi di Bari Aldo Moro,\\
Via E. Orabona 4, 70125 Bari, Italy.\\
\textit{email: d.dipinto10@studenti.uniba.it}\\

Antonio Lotta \\
Dipartimento di Matematica\\ Universit\`a degli Studi di Bari Aldo Moro,\\
Via E. Orabona 4, 70125 Bari, Italy.\\
\textit{email: antonio.lotta@uniba.it}\\
\end{document}